\documentclass[reqno,12pt]{amsart}
\usepackage{stmaryrd,mathrsfs,bm,amsthm,mathtools,yfonts,amssymb,color}
\usepackage{amsmath,a4wide}
\usepackage[colorinlistoftodos]{todonotes}

\numberwithin{equation}{section}

\theoremstyle{plain}
\newtheorem{theorem}[equation]{Theorem}
\newtheorem{prop}[equation]{Proposition}

\newtheorem{lemma}[equation]{Lemma}

\theoremstyle{definition}
\newtheorem{defn}[equation]{Definition}

\theoremstyle{remark}

\newcommand\R{{\mathbb R}}
\newcommand\re{\mathrm {Re}\,}
\newcommand\im{\mathrm {Im}\,}
\newcommand\genus{{\textsl{g}}}
\def\a#1{\left\llbracket{#1}\right\rrbracket}
\newcommand{\bB}{{\mathbf B}}
\newcommand{\mass}{{\mathbf{M}}}

\newcommand{\eps}{\varepsilon}

\DeclareMathOperator{\graph}{graph}
\DeclareMathOperator{\Ker}{Ker}
\DeclareMathOperator{\Id}{Id}


\begin{document}

\title[Nonclassical minimal surfaces]{Nonclassical minimizing surfaces with smooth boundary}
\author[De Lellis]{Camillo De Lellis}
\address{School of Mathematics, Institute for Advanced Study, 1 Einstein Dr., Princeton NJ 05840, USA\\
and Universit\"at Z\"urich}
\email{camillo.delellis@math.ias.edu}

\author[De Philippis]{Guido De Philippis}
\address{SISSA Via Bonomea 265, I34136 Trieste, Italy}
\email{gdephili@sissa.it}

\author[Hirsch]{Jonas Hirsch}
\address{Mathematisches Institut, Universit\"at Leipzig, Augustusplatz 10, D-04109 Leipzig, Germany}
\email{hirsch@math.uni-leipzig.de}

\begin{abstract}
We construct a Riemannian metric \(g\) on \(\mathbb{R}^4\) (arbitrarily close to the euclidean one)  and a smooth simple closed curve \(\Gamma\subset \mathbb R^4\) such that the unique area minimizing surface spanned by \(\Gamma\) has infinite topology. Furthermore the metric is  almost K\"ahler and the area minimizing surface is calibrated.
\end{abstract}

\maketitle

\section{Introduction} Consider a smooth closed simple curve $\Gamma$ in $\mathbb R^n$. The existence of oriented surfaces which bound $\Gamma$ and minimize the area can be approached in two different ways. Following the classical work of Douglas and Rado we can fix an abstract connected smooth surface $\Sigma_\genus$ of genus $\genus$ whose boundary $\partial \Sigma_\genus$ consists of a single connected component and look at smooth maps $\Phi: \Sigma_\genus \to \mathbb R^n$ with the property that the restriction of $\Phi$ to $\partial \Sigma_\genus$ is an homeomorphism onto $\Gamma$. We then consider the infimum $A_\genus (\Gamma)$ over all such $\Phi$ and all Riemannian metrics $h$ on $\Sigma$ of  
\[
\int_{\Sigma_\genus} |\nabla \Phi|^2\,{\rm dvol}_h\, .
\]
If $A_\genus (\Gamma) < A_{\genus -1} (\Gamma)$, then there is a minimizer $(\Phi, h)$ and the image of $h$ is an immersed surface of genus $\genus$, with possible branch points, see \cite{Douglas,Shiffman39,Courant40} and also \cite{Jost85,TomiTromba88}. 
The second, more intrinsic, approach was pioneered later by De Giorgi, in the codimension $1$ case~\cite{De-Giorgi55}, and by Federer and Fleming in higher codimension~\cite{FF}. They look at a suitable measure-theoretic generalization of smooth oriented surfaces, called integral currents $T$, whose generalized boundary is given by $\a{\Gamma}$ and minimize a suitable generalization of the area, called mass. In this framework a minimizer always exist and competitors do not have any topological restriction. 

\medskip

A basic question is whether the Federer-Fleming solution $T$ coincides with the Douglas-Rado solutions for some genus $\genus$. This is true if the curve $\Gamma$ is sufficiently regular ($C^{k,\alpha}$ for $k+\alpha>2$, because combining De Giorgi's interior regularity theorem  \cite{DG} with Hardt and Simon's boundary regularity theorem \cite{HS}, we know that every minimizer is an embedded $C^2$ surface up to the boundary $\Gamma$, in particular it has finite genus $\genus_0$. As corollaries, any conformal  parametrization $\Phi$ of $T$ gives a minimizer in the sense of Douglas and Rado, while $A_\genus (\Gamma) = A_{\genus_0} (\Gamma)$ for every $\genus > \genus_0$. If we instead merely assume that \(\Gamma\) has finite length, Fleming showed in~\cite{Fleming56} that it is possible to have $A_{\genus+1} (\Gamma)<A_{\genus} (\Gamma)$ for \(\genus\) arbitrarily large, implying in particular that every integral current minimizer has infinite topology, see also~\cite{AlmgrenThurston77} for related phenomena.

In higher codimension, namely for $n\geq 4$, it is known that the minimizer $T$ is in general not regular, neither in the interior nor at the boundary. Concerning the interior regularity, it has been shown by Chang in \cite{Chang} that $T$ is smooth in $\R^n\setminus \Gamma$ up to a discrete set of singular branch points and self-intersections (we in fact refer to \cite{DSS1, DSS2, DSS3, DSS4} for a complete proof, as Chang needs a suitable modification of the techniques of Almgren's monumental monograph \cite{Alm} to start his argument, and the former has been given in full details in \cite{DSS3}). As a corollary we know therefore that for any point $p\not\in \Gamma$ there is a neighborhood $U$ in which $T$ is the union of finitely many topological disks. Nonetheless it is still an open problem whether ``globally'' such solutions $T$ have finite topology. So far this can be only concluded If \(\Gamma\) is of class $C^{k,\alpha}$ for $k+\alpha >2$ and lies in the boundary of a uniformly convex open set, because Allard's boundary regularity theorem \cite{AllB} rules out boundary singularities.  

\medskip

In the general case, however, very little is known about the boundary regularity of area minimizing integral currents. The first result has been established by the authors and A. Massaccesi in the recent work \cite{DDHM}, which shows that, if $\Gamma$ is of class $C^{k,\alpha}$ for $k+\alpha >3$, then the set of regular boundary points is open and dense in \(\Gamma\). On the other hand the same paper gives a smooth simple closed curve $\Gamma$ in $\R^4$ bounding a (unique) minimizer $T$ which has infinitely many singularities. Such $T$ is, however, still an immersed disk, which has a countable number of self-intersections accumulating towards a boundary branch point: it is, in particular, a Douglas-Rado solution with genus $\genus = 0$.

In his work \cite{White97} White conjectures that the Federer-Fleming solution has finite genus if $\Gamma$ is real analytic. If White's conjecture were true, then the main theorem in \cite{White97} would imply that, for real analytic $\Gamma$, the set of boundary and interior singular points is finite and it would also exclude the presence of branch points at the boundary: the (finitely many) singular boundary points would all arise as self intersections.

\medskip

As already mentioned, the example in \cite{DDHM} shows that the latter conclusion would certainly be false for smooth $\Gamma$ in $\R^4$. In this note we show that, if we perturb the Euclidean metric in an appropriate way, the same curve bounds a unique area minimizing integral current with infinite topology. In particular, if we look at White's conjecture in Riemannian manifolds, real analyticity is a necessary assumption to exclude infinite topology of the Federer-Fleming solution. Our precise theorem is the following, where we denote by $\delta$ the standard Euclidean metric.

\begin{theorem}\label{t:main} For every $\varepsilon >0$ and every $N\in \mathbb N$ there is a smooth metric $g$ on $\R^4$, a smooth oriented curve $\Gamma$ in the unit ball $\bB_1$ passing through the origin and a smooth oriented surface $\Sigma$ in $\bB_1 \setminus \{0\}$ such that:
\begin{itemize}
\item[(a)] $g = \delta$ on $\R^4\setminus \bB_1$ and $\|g-\delta\|_{C^N} < \varepsilon$;
\item[(b)] $\a{\Sigma}$ is the unique area minimizing integral current in the Riemannian manifold $(\R^4, g)$ which bounds $\a{\Gamma}$;
\item[(c)] $\Sigma$ has infinite topology. 
\end{itemize}
\end{theorem}

In our example $\Sigma$ has (only) one singularity at the origin. The latter is a boundary singular point and $\Sigma$ displays a sequence of interior necks accumulating to it. A simple modification of our proof gives the existence of an area-minimizing current which bounds a smooth curve in a smooth Riemannian manifold and has an infinite number of interior branch points accumulating to the boundary. For the precise statement see Theorem \ref{t:branching} below. For the proofs of both Theorem \ref{t:main} and Theorem \ref{t:branching} it is essential that we are allowed to perturb the Euclidean metric. In particular the question whether such examples can exist is in some Euclidean space remains open.

\medskip

As pointed out, the question of whether the Federer-Fleming solution coincides with a Douglas-Rado solution is closely related to the regularity theory for area minimizers. We therefore close this introduction with a brief (and certainly not exhaustive) review of what is known for the Douglas-Rado solution. Interior branch points can be excluded in codimension \(1\), i.e. for surfaces in \(\R^3\), see~\cite{Osserman70, Alt72,Alt73,Gulliver73} and the discussion in~\cite[Section 6.4]{DierkesHildebrandtTromba10}. In higher codimension both interior branch points and self intersections are possible (primary examples are holomorphic curves in $\mathbb C^k = \mathbb R^{2k}$).
Concerning boundary branch points, it is well known that they can exist in higher codimension if the boundary curve is just $C^k$. The example of \cite{DDHM} mentioned above shows that they can exist even if it is $C^\infty$, while the aforementioned paper of White \cite{White97} excludes their existence when $\Gamma$ is real analytic. In fact the same conclusion was drawn much earlier in codimension \(1\) by a classical paper of Gulliver and Lesley,~\cite{GulliverLesley73}.

In codimension $1$ the existence of boundary branch points for the Douglas Rado solution is still an open question and it is probably the most important one in the field, we refer again to 
the discussion in~\cite[Section 6.4]{DierkesHildebrandtTromba10} for a detailed account of the known results. In~\cite{Gulliver91} Gulliver provides an interesting example of a  \(C^\infty\) curve in \(\mathbb R^3\) which bounds a minimal disk with one boundary branch point, however it is not known wether this surface is a Douglas-Rado solution. We note in passing that Gulliver's proof gives as well a Douglas-Rado disk-type solution (in fact a Federer Fleming solution) in \(\R^6\) spanning a \(C^\infty\) curve and with a boundary branch point.

\subsection*{Acknowledgements} The authors would like to thank Claudio Arezzo and Emmy Murphy for several interesting discussions. The work of G.D.P is supported by the INDAM grant ``Geometric Variational Problems''.

\section{Preliminaries}\label{sec:prel}

The Riemannian manifold $(\R^4, g)$ of Theorem \ref{t:main} has in fact a very special geometric structure, since it is an almost K\"ahler manifold.

\begin{defn}\label{d:Kaehler}
An almost complex structure on a smooth $4$-dimensional manifold $M$ is given by a smooth $(1,1)$ tensor $J$ with the property that $J^2 = - \Id$. The structure is almost K\"ahler if there is a smooth Riemannian metric $g$ with the properties that:
\begin{itemize}
\item[(i)] $J$ is isometric, namely $g (JV, JW) = g (V,W)$ for every vector fields $V$ and $W$;
\item[(ii)] The $2$-form defined by $\omega (V, W) := -g (V, JW)$ is closed. 
\end{itemize}
$\omega$ will be called the almost K\"ahler form associated to the almost K\"ahler structure.
\end{defn}

Theorem \ref{t:main} will then be a corollary of the following

\begin{theorem}\label{t:main2}
For every $\varepsilon >0$ and every $N\in \mathbb N$ there is a smooth metric $g$ on $\R^4$, a smooth oriented curve $\Gamma$ in the unit ball $\bB_1$ passing through the origin and a smooth oriented surface $\Sigma$ in $\bB_1 \setminus \{0\}$ such that:
\begin{itemize}
\item[(a)] $g = \delta$ on $\R^4\setminus \bB_1$ and $\|g-\delta\|_{C^N} < \varepsilon$;
\item[(b1)] there is an almost complex structure $J$ for which Definition \ref{d:Kaehler}(i)\&(ii) hold;
\item[(b2)] $\a{\Sigma}$ bounds $\a{\Gamma}$ and the pull-back of the corresponding $\omega$ on $\Sigma$ is the  volume form with respect to the metric \(g\); 
\item[(c)] $\Sigma$ has infinite topology.
\end{itemize}
\end{theorem}

Property (b2) is usually referred to as $\omega$ calibrating the surface $\Sigma$. 
It is a classical elementary, yet powerful, remark of Federer that the conditions (b1)-(b2) imply, by an inequality of Wirtinger, the minimality of the current $\a{\Sigma}$, cf. \cite{Fed}. Wirtinger's theorem shows that 
\[
\omega (V,W) \leq 1
\]
whenever 
\[
|V\wedge W|_g :=\sqrt{g (V,V) g (W,W) - g(V,W)^2} \leq 1
\] 
and that the equality holds if and only if $W = JV$. In the language of geometric measure theory Wirtinger's inequality implies that the comass (relative to the metric $g$) of the form $\omega$ is $1$. Moreover, we infer from the second part of Wirtinger's Theorem (the characterization of the equality case) that $\omega$ is pulled back to the standard volume form on $\Sigma$ if and only if there is a positively oriented tangent frame of the tangent bundle to $\Sigma$ of the form $\{V, JV\}$. Consider now any current (not necessarily integral!) $T$ which bounds $\a{\Gamma}$. Since $\omega$ is closed  and $\R^4$ has trivial topology, \(\omega\) has a primitive $\alpha$. We then must have
\[
T (\omega ) = T (d\alpha) = \int_\Gamma \alpha = \int_\Sigma d\alpha = \int_\Sigma \omega = \mass (\a{\Sigma})\, .
\] 
On the other hand Wirtinger's inequality implies that the comass of $\omega$ in the metric $g$ is $1$ and thus the mass of $T$ is necessarily larger than $T (\omega)$. 

This shows that $\a{\Sigma}$ is area minimizing. In order to conclude that it is the unique minimizer, we must appeal to the boundary regularity theory developed in \cite{DDHM}. First of all observe that, by \cite[Theorem 2.1]{DDHM} the interior regular set $\Lambda := {\rm Reg}_i (T)$ of the current $T$ is connected, it is an orientable submanifold of $\mathbb R^4$ and (up to a change of orientation) $T = \a{\Lambda}$. Moreover, by \cite[Theorem 1.6]{DDHM} there is at least one point $p\in \Gamma\setminus \{0\}$ and a neighborhood $U$ of $p$ such that $\Lambda \cap U$ is a smooth oriented surface with smooth oriented boundary $\Gamma \cap U$. By the argument above we must have $T (\omega) = \mass (T)$ and this implies, by Wirtinger's Theorem, that the tangent planes to $\Lambda$ are invariant under the action of $J$. The same holds for the tangent planes to $\Sigma$. In particular, the tangents to $\Sigma$ and $\Lambda$ must coincide at every point $q\in \Gamma \cap U$ and they must have the same orientation. Since both are smooth minimal surfaces in $U$, the unique continuation for elliptic systems implies that they coincide in a neighborhood of $q$. Again, thanks to  the unique continuation principle  and the connectedness of $\Lambda$ we conclude that $\Lambda$ is in fact a subset of $\Sigma$. However, since they have the same area, this implies that $\a{\Sigma} = T$. 

\section{Proof of Theorem \ref{t:main2}: Part I}\label{s:complex}

In this section we slightly modify the construction given in \cite[Section 2.3]{DDHM} to achieve a smooth curve $\Gamma$ in $\R^4$ and an integral current $T$ in $\R^4$ such that
\begin{itemize}
\item[(i)] $T$ bounds $\a{\Gamma}$ and is area minimizing in $(\R^4, \delta)$ (i.e. with respect to the Euclidean metric), in fact $T$ is induced by an holomorphic subvariety in $\R^4\setminus \Gamma$;
\item[(ii)] $T$ is regular at $\Gamma\setminus \{0\}$;
\item[(iii)] $0$ is an accumulation point for the interior singular set of $T$, denoted by ${\rm Sing}_i (T)$;
\item[(iv)] At each $p\in {\rm Sing}_i (T)$ there is a neighborhood $U$ such that $T$ in $U$ consists of two holomorphic curves intersecting transversally at $p$.
\end{itemize}

First of all consider the complex plane with an infinite slit 
\[
\mathbb K := \{z\in \mathbb C\} \setminus \{z\in \mathbb R : z\leq 0\}\, .
\]
We consider the usual inverse $\arctan: \mathbb R \to (-\frac{\pi}{2}, \frac{\pi}{2})$ on the real axis of the trigonometric function $\tan$ and we fix a determination of the complex logarithm on $\mathbb K$ which coincides with
\[
{\rm Log}\, z = \log |z| + i \arctan \frac{{\rm Im}\, z}{{\rm Re}\, z}\, .
\]
on the open half plane $\mathbb H := \{z\in \mathbb C : \re z > 0\}$.
Correspondingly we define the functions $z^{-\alpha} = \exp (-\alpha {\rm Log}\, z)$ for
$\alpha\in (0,1)$ and
\[
f_k (z) = \exp (- z^{-\alpha}) \sin \left({\rm Log}\, z + \frac{3-2k}{6} \pi i \right)\,  \qquad \mbox{for $k=0,1,2,3$.}
\]
Observe that:
\begin{itemize}
\item[(i)] If we extend each $f_k$ to the origin as $0$, then $f_k$ is a smooth function over any wedge 
\[
\mathbb K_a := \{z: - \re z \leq a |\im z|\}
\] 
with $a$ positive.   
\item[(ii)] Since $\exp (- z^{-\alpha})$ does not vanish on $\overline{\mathbb H}\setminus \{0\}$, the zero set $Z_k$ of $f_k$ in $\overline{\mathbb H}\setminus \{0\}$ is given
by 
\[
Z_k= \left\{z\in \overline{\mathbb H}: {\rm Log}\, z + \frac{3-2k}{6} \pi i \in \pi \mathbb Z\right\}\, ,
\]
namely by 
\begin{equation}\label{e:formula_Z_k}
Z_k = \left\{\exp \left(n \pi + i\frac{2k-3}{6}\pi \right): n \in \mathbb Z\right\}\, . 
\end{equation}
\end{itemize}
Consider next the function
\[
g(z) = \prod_{k=0}^3 f_k (z)\, .
\]
We then conclude that $g$ is holomorphic on $\mathbb K$, it is $C^\infty$ on $\mathbb K_a$ for every $a>0$ and its zero set in $\overline{\mathbb H}$, which we denote by $Z$, is given by
\[
Z = \{0\} \cup \bigcup_{k=0}^3 Z_k\, .
\]
Define now the map $G: \mathbb K\to \mathbb C^2$ by $G(z) = (z^3, g(z))$. We consider a smooth simple
curve $\gamma\subset \mathbb K_1$ which in a neighborhood of the origin is tangent to the imaginary axis 
and we let $D\subset \mathbb K_1$ be the open disk bounded by $\gamma$. Following the arguments of \cite[Section 2.3]{DDHM} it is not difficult to see that $\gamma$ can be chosen so that:
\begin{itemize}
\item[(A)] $\{g=0\} \cap \gamma = \{0\}$;
\item[(B)] $\{g=0\}\cap D\subset \overline{\mathbb H}$, hence $\{g=0\}\cap D \subset Z$ and, for each $k\in \{0, \ldots , 3\}$, it contains all sufficiently small elements of $Z_k$, namely there is a positive constant $c_0$ such that $\{z\in Z_k : |z|\leq c_0\} \subset D$.
\end{itemize}
The current $T := G_\sharp \a{D}$ is integer rectifiable, it has multiplicity one (in particular it coincides with $\a{G (D)}$) and 
\[
\partial T = G_\sharp \partial \a{D} = G_\sharp \a{\gamma}\, .
\] 
Observe that $G (D)$ is an holomorphic curve of $\mathbb C^2$, which carries a natural orientation. If $\a{G(D)}$ denotes the corresponding integer rectifiable current, we then can follow the argument in \cite[Section 2.3]{DDHM} to show that $T = \a{G(D)}$ and Federer's classical argument implies that $T$ is area minimizing for the standard Euclidean metric.

The arguments given in \cite[Section 2.3]{DDHM} show that $G_\sharp \a{\gamma} = \a{G (\gamma)}$ and $G (\gamma) \subset \mathbb C^2 = \mathbb R^4$ is a smooth embedded curve. The same arguments also show that $G (D)$ is a smooth immersed surface, that it is embedded outside the discrete set $G (Z)$ and that at each point $q\in G (Z\cap D)$ it consists of two holomorphic graphs intersecting transversally.

\section{Proof of Theorem \ref{t:main2}: Part II}\label{s:desingularization}

In order to conclude the proof of Theorem \ref{t:main} the idea is to modify the example of the previous section and substitute the self-intersection of each singular point $q\in G(Z\cap D)$ with a neck. In order for the new surface to be area minimizing we will then perturb the Euclidean metric and the standard complex structure to a nearby metric and a nearby almost K\"ahler structure. More precisely, order the points $\{p_k\}_{k\in \mathbb N}$ of the discrete set $G (Z\cap D)$. Fix sufficiently small balls $\bB_{100r_k} (p_k)$ so that they are all disjoint and do not intersect the boundary curve $\Gamma$. Recall that $G (D) \cap \bB_{100 r_k} (p_k)$ consist of two holomorphic disks intersecting transversally at $p_k$. In particular, we can assume that the two tangents to these disks are given by $\pi_1$ and $\pi_2$, where $\pi_1$ and $\pi_2$ are two distinct affine complex planes, namely 
\begin{align}
\pi_1 &= p_k + \{(z,w): a_1 z+ b_1 w =0\}\\
\pi_2 &=p_k + \{(z,w) : a_2 z+ b_2 w =0\}
\end{align}
for two different points $[a_1, b_1], [a_2, b_2] \in \mathbb C \mathbb P^1$.
The idea is to choose a sufficiently small $\eta_k>0$ and substitute the surface $G(D)$ inside $\bB_{r_k} (p_k)$ with the holomorphic subvariety 
\[
\Lambda_k := \{p_k + (z,w): (a_1z + b_1 w) (a_2z + b_2 w)  = \eta_k\}\, ,
\] 
while glueing it back to the original surface $G(D)$ in the annulus $\bB_{100 r_k} (p_k)\setminus \overline{\bB}_{r_k} (p_k)$.

If $\eta_k$ and $r_k$ are sufficiently small, we can assume that $\Lambda_k \cap \bB_{100r_k} (p_k) \setminus \overline{\bB}_{r_k} (p_k)$ and $G (D) \cap \bB_{100r_k} (p_k) \setminus \overline{\bB}_{r_k} (p_k)$ consist each of two annuli, respectively $\Lambda^1_k$, $\Lambda^2_k$ and $\Sigma^1_k$, $\Sigma^2_k$, where $\Lambda^i_k$ is close to $\Sigma^i_k$. Moreover, again by assuming that $\eta_k$ and $r_k$ are sufficiently small, each $\Lambda^i_k$ and $\Sigma^i_k$ are graphs of holomorphic functions  over the plane $p_k +\pi_i$. We now wish to glue the surfaces $\Sigma^1_k$ and $\Lambda^1_k$ and $\Sigma^2_k$ and $\Lambda^2_k$ and modify the Euclidean metric and the standard K\"ahler structure in the annulus $\bB_{100r_k} (p_k) \setminus \overline{\bB}_{r_k} (p_k)$ to a nearby Riemannian metric with a corresponding almost K\"ahler structure, so that the glued surface is calibrated by the associated almost K\"ahler form. Both the new metric and the corresponding 
almost K\"ahler form will coincide with the Euclidean metric and the standard K\"ahler form outside of a neighborhood of the glued surface. By assuming $r_k$ and $\eta_k$ very small, we can reduce to perform such glueing in neighborhoods of the planar annuli $(p_k +\pi_1) \cap  \bB_{100r_k} (p_k) \setminus \overline{\bB}_{r_k} (p_k)$ and $(p_k +\pi_2) \cap  \bB_{100r_k} (p_k) \setminus \overline{\bB}_{r_k} (p_k)$, which are disjoint. In particular we can assume that we glue the two pairs of surfaces and we modify the metric and the K\"ahler form in two separate regions. A schematic picture summarizing our discussion is given in the picture below.

\begin{figure}[htbp]
\begin{center}\label{fig:1}
\includegraphics{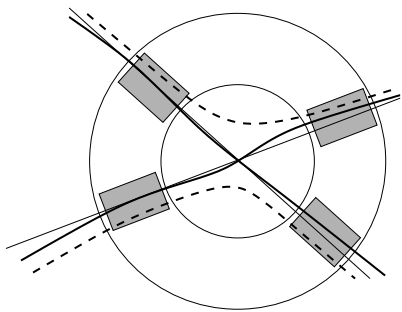}
\end{center}
\caption{A schematic picture of the procedure outlined above. The figure contains cross sections of the corresponding objects with the real affine plane $p_k + \mathbb R \times \mathbb R \subset \mathbb C \times \mathbb C$. In particular, $G (D)$ is pictured by the thick continuous curves, which in $p_k$ are tangent to the union of two crossing complex lines $p_k + \pi_1$ and $p_k + \pi_2$. The dashed lines represent the hyperbola $\Lambda_k$. The surface 
$\Sigma$ will coincide with the dashed lines in the inner ball, with the thick lines outside the outer ball and with a smooth interpolation between the two surfaces in the annular region. The interpolation will take place in the shadowed region, where both $\Lambda_k$ and $G(D)$ are graphical over the corresponding portion of $p_k + \pi_i$.}
\end{figure}

The corresponding metric $g_k$ will coincide with the euclidean one outside of the annulus and will have the property that, if we set $\bar{k} := \max \{k, N\}$, then 
\begin{equation}\label{e:geometric}
\|g_k - \delta\|_{C^{\bar k}} < \varepsilon 2^{-\bar k-1}\, ,
\end{equation}
where $\varepsilon$ is the constant of Theorem \ref{t:main2}. 
The latter estimate will be achieved by choosing $\eta_k$ appropriately small, so that the graphs $\Lambda^i_k$ almost coincide with the graphs $\Sigma^i_k$.

The surface $\Sigma$ and the metric $g$ of Theorem \ref{t:main2} will then be defined as follows:
\begin{itemize}
\item Outside of $\bigcup_k \bB_{100 r_k} (p_k)$ $\Sigma$ coincides with $G (D)$ and the metric $g$ is the Euclidean metric.
\item Inside each $\overline{\bB_{r_k} (p_k)}$ $\Sigma$ coincides with the holomorphic submanifold $\Lambda_k$ and the metric $g$ is the Euclidean metric.
\item In the annulus $\bB_{100r_k} (p_k) \setminus \overline{\bB}_{r_k} (p_k)$ $\Sigma$ is the glued surface and $g$ is the metric $g_k$ described above.
\end{itemize}
The existence of the (local) glued surface and of the metric $g_k$ is thus the key point and is guaranteed by the glueing proposition below (after appropriate rescaling). In the rest of the note we use the following notation:
\begin{itemize}
\item $D_r (p) \subset \mathbb C$ is the disk centered at $p\in \mathbb C$ of radius $r$; $p$ will be omitted if it is the origin.
\item $\omega_0$ is the K\"ahler form on $\mathbb R^4 = \mathbb C^2$ and  $\delta$ is the Euclidean metric on $\mathbb R^4$.
\item $J_0$ is the standard complex structure on $\mathbb R^4$, namely $J_0 (a,b,c,d) = (-b,a,-d,c)$.
\item Norms on functions, tensors, etc. are computed with respect to the Euclidean metric.
\end{itemize}

\begin{prop}[Glueing]\label{l:glueing}
For every $\eta>0, N\in \mathbb N$ there is $\varepsilon >0$ with the following property. Assume that 
$f, h: D_{10} \setminus \overline{D}_1 \to \mathbb C$ are two holomorphic maps with 
\begin{equation*}
\|f\|_{C^{N+2}} +\|h\|_{C^{N+2} }\leq \varepsilon\, .
\end{equation*}
Then there are
\begin{itemize}
\item[(i)] a metric $g\in C^\infty$ with $\|g-\delta\|_{C^N} \leq \eta$ and $g=\delta$ outside $(D_9 \setminus \overline{D}_1)\times D_{2\eta}$,
\item[(ii)] an almost K\"ahler structure $J$ compatible with $g$ such that $\|J-J_0\|_{C^N} \leq \eta$ and $J=J_0$ outside $(D_9 \setminus \overline{D}_1)\times D_{2\eta}$,
\item[(iii)] an associated  almost K\"ahler form $\omega$ with  $\|\omega-\omega_0\|_{C^N} \leq \eta$ and $\omega=\omega_0$ outside $(D_9 \setminus \overline{D}_1)\times D_{2\eta}$
\item[(iii)] and a function $\zeta: D_{10} \setminus \overline{D}_1 \to D_\eta$ 
\end{itemize}
such that
\begin{itemize}
\item[(a)] $\zeta = (\re f, \im f)$ on $D_2\setminus \overline{D_1}$ and $\zeta = (\re h, \im h)$ on $D_{10}\setminus \overline{D_9}$;
\item[(b)] $\omega$ calibrates the graph of $\zeta$. 
\end{itemize}  
\end{prop}

\section{Proof of the glueing proposition} 

Before coming to the proof, let us recall some known facts from symplectic geometry. First of all, a $2$-form $\alpha$ on $\mathbb R^{2n}$ is called nondegenerate if for every point $p\in \mathbb R^{2n}$ the corresponding skew-symmetric bilinear map $\alpha_p : \mathbb R^{2n}\times \mathbb R^{2n} \to \mathbb R$ is nondegenerate, namely
\begin{equation}\label{e:nondegenerate}
\forall u \in \mathbb R^{2n}\neq 0\; \exists v\in \mathbb R^{2n} \mbox{ with } \alpha_p (u,v)\neq 0\, . 
\end{equation}
Given a skew-symmetric form we can define \(A_p: \R^{2n}\to \R^{2n}\) as 
\begin{equation}\label{e:defA}
\omega_p(v,w)=-\delta(v,A_pw)
\end{equation}
where we recall that \(\delta\) is the Euclidean metric. The nondegeneracy condition \eqref{e:nondegenerate} is  equivalent to $\ker A_p = \{0\}$. Note that if \(\omega_0\) is the standard K\"ahler form of \(\R^{2n}\), then \(A_p=J_0\) and that 
\[
\|\alpha-\omega_{0}\|_{C^N}\le \eta \Rightarrow \|A_p-J_0\|_{C^N}\le C\eta
\]
In particular, any $2$-form which is sufficiently close to $\omega$ in the $C^0$ norm is necessarily nondegenerate.

We start with  the following particular version of the Poincar\'e Lemma. Since we have not been able to find a precise reference, we give the explicit argument.

\begin{lemma}\label{l:Poincare}
Assume $U\subset \R^4$ is a star-shaped domain with respect to the origin and let $\beta$ be a closed $2$-form, with the property that the pull back of $\beta$ on $\{X_3=X_4=0\}$ vanishes. Then there is a primitive $\alpha$ with the properties that
\begin{itemize}
\item $\alpha$ vanishes identically on $\{X_3=X_4=0\}$;
\item $\|\alpha\|_{C^N} \leq C \|\beta\|_{C^{N+1}}$, where the constant $C$ depends only on the diameter of $U$.
\end{itemize}
\end{lemma} 
\begin{proof}
First of all recall the standard formula for the primitive of a form given by integration along rays (cf. \cite[Theorem 4.1]{Spivak}). Namely, if
\[
\bar \beta = \sum_{i< j} \bar \beta_{ij}\, dX_{i} \wedge dX_{j}\, ,
\]
then a primitive $\bar \alpha$ can be computed using the formula
\begin{equation}\label{e:Spivak}
\bar \alpha (X) = \sum_{i} \sum_{j} (-1)^{j-1} \int_0^1 t \bar \beta_{ij} (tX)\, dt\, X_{j} dX_{i}
\end{equation}
with the  convention that $\bar \beta_{ij} = -\bar  \beta_{ji}$ if \(i>j\).
Using the latter expression we obviously have $\|\bar \alpha\|_{C^N} \leq C \|\bar \beta\|_{C^N}$. Moreover, if $\bar \beta$ vanishes identically on $\{X_3=X_4=0\}$ then clearly $\bar \alpha$ vanishes identically on $\{X_3=X_4=0\}$.

Given a general closed $2$-form $\beta$, we then look for a $1$-form $\vartheta$ which vanishes on $\{X_3=X_4 =0\}$ and with the property that $\bar \beta := \beta - d\vartheta$ vanishes on $\{X_3=X_4=0\}$. The resulting $\alpha$ will then be found as $\bar{\alpha} + \vartheta$, where $\bar\alpha$ is the primitive of $\bar \beta$ given in the formula \eqref{e:Spivak}. In order to find $\vartheta$ we first write $\beta$ in the form
\[
\beta = f\, dX_1\wedge dX_2 + \underbrace{(a_1  dX_1 + a_2\, dX_2)}_{=:\lambda} \wedge dX_3 + \underbrace{(b_1\, dX_1 + b_2\, dX_2 + b_3\, dX_3)}_{=:\mu} \wedge dX_4\, .
\]
By assumption $f$ equals $0$ on $\{X_3=X_4=0\}$. Let us set 
\[
\vartheta = -X_3 \mu -X_4 \lambda\, ,
\]
so that
\[
\beta -d\vartheta= f\, dX_1\wedge dX_2 + X_3 d\mu + X_4 d\lambda\, .
\]
Since $f$ vanishes on $\{X_3=X_4=0\}$ we then get the desired property that $\beta-d\vartheta$ vanishes on it as well.
\end{proof}

\begin{proof}[Proof of the Glueing Proposition] We will focus on the construction of the triple, whereas the estimates are a simple consequence of the algorithm.

\medskip
\noindent
\emph{Step 1: Definition of \(\zeta\) and a new  system of coordinates}: First we smoothly extend \(f\) inside \(D_1\) and we then define \(\zeta\) as 
\[
\zeta = (\re f, \im f) \varphi + (\re h, \im h) (1-\varphi)\, .
\]
where $\varphi \in C^\infty_c (D_5)$ with $0\leq \varphi \leq 1$ and $\varphi\equiv 1$ on $D_4$. In particular 
\[
\zeta=f \qquad\text{on \(D_4\)}\qquad\text{and}\qquad\zeta=h \qquad\text{outside \(D_5\)}.
\]
 We now choose  a system of coordinates \(X:=(X_1,\dots, X_4)\) such that  \(\|X-\Id\|_{C^{N+1}}\le C\eps\),
 \begin{equation}\label{e:system1}
 \Sigma=\mathrm{graph}(\zeta)=\{X_3=X_4=0\}\,
\end{equation}
and
 \begin{equation}\label{e:system2}
T_{p} \Sigma=\Ker dX_3\cap \Ker dX_4\qquad T_{p} \Sigma^\perp=\Ker dX_1\cap \Ker dX_2
\end{equation}
Note that this can be done by, for instance, taking normal coordinates around \(\Sigma\), provided \(\eps\) is chosen sufficiently small. 

More precisely, we first choose two vector fields \(\xi, \tau\) along \(\Sigma\) such that:
\begin{itemize}
\item \(|\xi_p|=|\tau_p|=1\) and \(\xi_p\perp \tau_p\) (in the euclidean metric);
\item \(T_{p}\Sigma^{\perp}=\mathrm{span} (\xi_p, \tau_p)\).
\end{itemize}   
We set
\[
Y(x_1,x_2, x_3, x_4)=(x_1,x_2, \zeta_1(x_1,x_2), \zeta_2(x_1,x_2))+x_3 \xi(x,_1,x_2)+x_4 \tau(x_1,x_2). 
\]
where \(\zeta=(\zeta_1, \zeta_2)\),
\[
p=(x_1,x_2, \zeta_1(x_1,x_2), \zeta_2(x_1,x_2))\in \Sigma, \qquad \xi(x,_1,x_2)=\xi_p \qquad {\color{blue}\tau}(x_1, x_2)={\color{blue}\tau}_p.
\]
In order to get vector fields $\zeta$ and $\tau$ whose derivatives are under control, a standard procedure is to take the standard vector fields $e_3 = (0,0,1,0)$, $e_4 = (0,0,0,1)$, project them orthogonally onto $T_p^\perp \Sigma$ and apply the Gram-Schmidt orthogonalization procedure to them. Simple computations give that
\[
\|\xi-e_3\|_{C^{N+1}} + \|\tau-e_4\|_{C^{N+1}} \leq C \|\zeta\|_{C^{N+2}}\, .
\]
Note in particular that, if \(\eps\) is chosen sufficiently small,  \(Y\) is a diffeomorphism onto its image and that the latter  contains \(D_{8}\times D_{8}\). Letting \(X=Y^{-1}\) it is immediate to check that \eqref{e:system1} is satisfied and thus also the first equality in  \eqref{e:system2}. To check the second one simply note, by the very definition of \(X\),
\[
\{X_1=c_1, X_2=c_2\}=p+T_{p} \Sigma^\perp \qquad X(p)=(c_1,c_2,0,0).
\]
 From now on, 
with a slight abuse of notation, we will denote by \(D_{r}\times D_{s}\) the product of disks in the \(X\) system of coordinates, that is
\[
D_{r}\times D_{s}=\{ X_1^2+X_2^2< r^2\,, \,X_3^2+X_4^2< s^2\}
\]
and we will work in the domain $D_8\times D_8$. 
Given that $\|DX - {\rm Id}\|_{C^0} + \|DY - {\rm Id}\|_{C^0} \leq \varepsilon$ and assuming, without loss of generality, that $X$ and $Y$ keep the origin fixed, such sets are comparable to the corresponding products $D^e_\rho \times D^e_\sigma$ in the euclidean system of coordinates, namely 
\[
D^e_{C^{-1} r} \times D^e_{C^{-1} s} \subset D_r\times D_s \subset D^e_{C r} \times D^e_{C s} 
\] 
where the constant $C$ approaches \(1\) as \(\eps\to 0\).

\medskip
\noindent
\emph{Step 2: Construction of the  \(2\) form}: We take  \(\sigma \ll \eta\) and, provided \(\eps\ll \sigma\),  we claim the existence of a  \(2\)-form \(\omega\) on \(D_{8}\times D_8\) such that 
\begin{enumerate}
\item[(a)] $\omega$ is closed (and hence exact);
\item[(b)] The pull back of \(\omega\) and    \(\omega_0\) are the same  on \(\Sigma\). 
\item[(c)]For all \(p\ \in \Sigma\cap\bigl( (D_{7}\setminus \overline{D_{3}})\times D_{8}\bigr)\) 
\begin{equation}\label{e:omegaspecial}
\omega_p=\omega_{p}(v,w)=0  \qquad \text{for all \(v\in T_{p} \Sigma\) and all \(w\in T_p\Sigma^\perp\) 
}
\end{equation} 
\item [(d)] \(\omega=\omega_0\) outside of $D_7\setminus \overline{D}_3\times D_{2\sigma}$%on $((D_8 \setminus \overline{D}_7)\cup (D_3\setminus \overline{D}_2))\times D_{8} \cup (D_7\setminus \overline{D}_3)\times (D_8\setminus \overline{D}_{2\sigma})$. 
\item[(e)] \(\|\omega-\omega_0\|_{C^N}\le \eta\).
\end{enumerate}
To construct the form we observe that, on \(\Sigma\),  $i_\Sigma^\sharp \omega_0=a(X_1,X_2) dX_1\wedge dX_2$ for a suitable smooth function \(a\). Extending \(a\) constant in the \(X_3,X_4\) coordinates we can write 
\begin{equation*}\label{e:omega0onsigma}
\omega_0=a(X_1, X_2)dX_1\wedge d X_2+\overline {\omega}
\end{equation*}
where \(\bar \omega\) is pulled back to \(0\) on \(\Sigma\). Note that \(d \bar \omega=0\) since \(\omega\) is closed. Moreover 
\begin{equation}\label{e:estimates}
\|a-1\|_{C^{N+1} (D_8)}+\|\overline {\omega}-dX_3\wedge dX_4\|_{C^{N+1} (D_8\times D_8)}\le o_\eps(1)
\end{equation}
where \( o_\eps(1)\to 0\) as \(\eps\to 0\). We define
\begin{equation}\label{e:beta}
\beta=a(X_1,X_2) dX_1\wedge dX_2+dX_3\wedge dX_4.
\end{equation}
We now apply Lemma \ref{l:Poincare} to find a primitive \(\vartheta\) of \(\omega_{0}-\beta=\overline {\omega}-dX_3\wedge dX_4\) which equals  \(0\) on \(\Sigma\).  We also let \(\varphi\) be a smooth cut-off function such that 
\[
\begin{cases}
\varphi\equiv 0 &\text{on  \((D_{6}\setminus D_{4})\times D_{\sigma}\)}.
\\
\varphi\equiv 1\qquad  &\text{outside   $(D_7 \setminus \overline{D}_3)\times D_{2\sigma}$}
\end{cases}
\]
Note in particular that, provided \(\eps\ll \sigma\), 
\begin{equation}\label{e:inclusione}
\Sigma\cap \{\varphi \neq 0\} \subset \graph f\cup \graph h\, .
\end{equation}
We define 
\[
\omega=\beta +d(\varphi \vartheta)=\beta+\varphi (\omega_0-\beta)+d\varphi\wedge \vartheta.
\]
Clearly \(\omega\) satisfies (a)  and (d). Property  (e) follows by choosing \(\eps\ll \sigma\) as a  consequence of the construction of Lemma \ref{l:Poincare} and of  \eqref{e:estimates}. Moreover since \(\vartheta\) vanishes on \(\Sigma\) and the pull-backs of \(\beta\) and \(\omega_0\) on \(\Sigma\)  are  the same, also (b) is satisfied. To check (c)  we note that due to \eqref{e:system2} and the definition of $\beta$ we have that $\beta_p(v,w)=0$ for all $p$ in the domain of $X$ and $v \in T_p\Sigma, w \in T_p\Sigma^\perp$. In particular \eqref{e:omegaspecial} is satisfied on \(\{\varphi=0\}\).   Since $f,h$ are holomorphic outside \(D_2\), by \eqref{e:inclusione}, for $p\in \{\varphi \neq 0 \}\cap (D_8\setminus \overline{D_2})\times D_{8}$ the spaces \(T_{p} \Sigma\) and \(T_{p}\Sigma^\perp\) are perpendicular complex lines. Hence \(\omega_0\) satisfies \eqref{e:omegaspecial} there, since 
\[
\omega\bigr|_{\Sigma}=(1-\varphi)\beta+ \varphi \omega_0,
\] 
\(\omega\) satisfies \eqref{e:omegaspecial}  as well and (c) is verified.
 
\medskip
\noindent
\emph{Step 3: Definition of the almost complex structure and of the metric}:  To conclude the proof it will be enough to construct a metric \(g\) and a compatible almost complex structure \(J\). Here we follow a method used in~\cite{Bellettini}. Let \(A_p\) be the skew-symmetric matrix defined in \eqref{e:defA}. In particular \(Q_p=-A_p^2 =A_pA_p^{t}\) defines a positive definite quadratic form and thus it admits a (positive definite) square root. We set 
\[
g_{p}=\big(-A_p^2\big)^{-\frac{1}{2}}\qquad\text{and}\qquad J_{p}=g_{p}^{-1}A_p.
\]
Note that \(g_{p}\) and \(J_p\) equal, respectively, the euclidean metric \(\delta\) and the usual complex structure \(J_0\) where \(\omega=\omega_0\). Furthermore, since \(g_p\) commutes with \(A_p\), one immediately verifies that 
\[
J_{p}^{2}=-\Id\qquad g_{p}(J_{p} v, J_{p} w)=g_{p}( v,  w)\qquad \omega_{p}(v,w)=-g_{p}( v, J_{p} w)
\]
so that the triple \((g, J, \omega)\) defines an almost K\"ahler  structure on \(D_8\times D_8\) which coincides with the canonical one where \(\omega=\omega_0\). We are thus left to prove that \(\omega\) calibrates \(\Sigma\) on \((D_{10}\setminus \overline{D_{1}})\times D_{10}\). This is clear in the region where \(\omega=\omega_0\),  because in that region \(\Sigma\) equals either the graph  of \(f\) or that of \(h\) and these are holomorphic outside \(D_1\). Hence it is enough to verify that $\Sigma$ is calibrated in \(\bigl( (D_{7}\setminus \overline{D_{3}})\times D_{8}\bigr)\). To this end note that, if \(p\in \Sigma\cap\bigl( (D_{7}\setminus \overline{D_{3}})\times D_{8}\bigr)\),  by \eqref{e:omegaspecial} and the definition of \(A_p\), 
\[
0=\omega_{p}(v,w)=-\delta(v,A_pw)=\delta(A_pv, w) \qquad \mbox{ for all  \(v\in T_{p}\Sigma\) , \(w\in T_{p}\Sigma^\perp\).}
\]
In particular \(A_p\) maps \(T_p\Sigma\) into itself (and \(T_{p} \Sigma^\perp\) into itself as well). The same is true then for \(J_p\)  and thus, if \(g_p(v,v)=1\), \((v, J_pv)\) is a \(g_p\)-orthonormal frame of \(T_{p}\Sigma\). This implies that  \(\omega\) is pulled back on \(\Sigma\) to the \(g\)-volume form and concludes the proof. 
\end{proof}

\section{Branching singularities}

A simple modifications of the ideas outlined above proves the following

\begin{theorem}\label{t:branching}
For every $\varepsilon >0$ and every $N\in \mathbb N$ there is a smooth metric $g$ on $\R^4$, a smooth oriented curve $\Gamma$ in the unit ball $\bB_1$ passing through the origin and a smooth oriented surface $\Sigma$ in $\bB_1 \setminus \{0\}$ such that:
\begin{itemize}
\item[(a)] $g = \delta$ on $\R^4\setminus \bB_1$ and $\|g-\delta\|_{C^N} < \varepsilon$;
\item[(b)] $\a{\Sigma}$ is the unique area minimizing integral current in the Riemannian manifold $(\R^4, g)$ which bounds $\a{\Gamma}$;
\item[(c')] There is an finite number of branching singularities $p_k\in \Sigma\setminus \Gamma$ accumulating to the only boundary singular point $0$. 
\end{itemize}
\end{theorem}

The idea of the proof is to produce the analogous to Theorem \ref{t:main2} where the conclusion (c) therein is substituted by the conclusion (c') above. Here we sketch the necessary modifications to the arguments given for Theorem \ref{t:main2}. 

We start by constructing an example of an holomorphic subvariety inducing an area minimizing current $T$ as in Section \ref{s:complex} where the property (iv) is however replaced by
\begin{itemize}
\item[(v)] At each $p\in {\rm Sing}_i (T)$ there is a neighborhood $U$ such that $T$ in $U$ consists of {\em four} holomorphic curves intersecting transversally at $p$.
\end{itemize}
More precisely there are four distinct elements $[a_1, b_1], [a_2, b_2], [a_3, b_3], [a_4, b_4] \in \mathbb C \mathbb P^1$ such that the tangent cone to $T$ at $p$ is given by the union of four corresponding complex lines:
\[
\left\{(z,w) : \prod_{i=1}^4 (a_i z + b_i w) =0 \right\}\, .
\] 
In order to achieve such object we construct a similar function $g$ as in Section \ref{s:complex}, by defining
\[
f_k (z) = \exp (-z^\alpha) \sin \left({\rm Log}\, z + \frac{7 - 2k}{14} \pi i \right) \qquad \mbox{for $k = 0, 1 \ldots , 7$}
\]
and
\[
g (z) = \prod_{k=0}^7 f_k (z)\, .
\]
We then proceed as in Section \ref{s:complex} to define the zero sets $Z_k$ of $f_k$ on $\overline{\mathbb H}\setminus \{0\}$, the set $Z = \{0\} \cup \bigcup_k Z_k$, the curve $\gamma$ and the corresponding disk $D$, where we require the properties analogous to (A) and (B) therein. We finally define the map $G (z) := (z^7, g (z))$ and the current $T$ is thus given by $\a{G(D)}$.

Next, proceeding as in Section \ref{s:desingularization}, in a sufficiently small ball of radius $r_k$ centered at $p_k\in {\rm Sing}_i (T)$ we wish to replace $G (D)$ with another holomorphic subvariety, which has a branching singularity at $p_k$. Since $G (D)$ is, at small scale, very close to the cone
\[
C_k :=\bigcup_{i=1}^4 \underbrace{\left\{p_k + (z,w) : (a_i z + b_i w) =0 \right\}}_{=:\pi_{k,i}}\, ,
\]
the idea is to choose
\[
\Lambda_k := \left\{p_k + (z,w) : \prod_{i=1}^4 (a_i z + b_i w) = \eta_k (z^3-w^2)\right\}\, ,
\]
where $\eta_k$ is again a very small parameter. Choosing $r_k$ and $\eta_k$ sufficiently small, we can ensure that $G (D) \cap \bB_{100 r_k} \setminus \overline{\bB}_{r_k} (p_k)$ and $\Lambda_k \cap \bB_{100 r_k} \setminus \overline{\bB}_{r_k} (p_k)$ consist each of four annuli which are graphs over corresponding annular regions of the four distinct complex lines $\pi_{k, i}$, $i=1, \ldots, 4$. We can obviously engineer such graphs to be arbitrarily close to the corresponding planes, and hence to fall, after appropriating rescaling under the assumption of the glueing Proposition \ref{l:glueing}. Hence the construction of $\Sigma$ and of the almost K\"ahler structure $(g, J, \omega)$ follows the same arguments.

%%
%\bibliographystyle{abbrv}
%\bibliography{references-infinite-topology}

\begin{thebibliography}{10}

\bibitem{AllB}
W.~K. Allard.
\newblock On the first variation of a varifold: boundary behavior.
\newblock {\em Ann. of Math. (2)}, 101:418--446, 1975.

\bibitem{Alm}
F.~J. Almgren, Jr.
\newblock {\em Almgren's big regularity paper}, volume~1 of {\em World
  Scientific Monograph Series in Mathematics}.
\newblock World Scientific Publishing Co. Inc., River Edge, NJ, 2000.

\bibitem{AlmgrenThurston77}
F.~J. Almgren, Jr. and W.~P. Thurston.
\newblock Examples of unknotted curves which bound only surfaces of high genus
  within their convex hulls.
\newblock {\em Ann. of Math. (2)}, 105(3):527--538, 1977.

\bibitem{Alt72}
H.~W. Alt.
\newblock Verzweigungspunkte von {$H$}-{F}l\"{a}chen. {I}.
\newblock {\em Math. Z.}, 127:333--362, 1972.

\bibitem{Alt73}
H.~W. Alt.
\newblock Verzweigungspunkte von {$H$}-{F}l\"{a}chen. {II}.
\newblock {\em Math. Ann.}, 201:33--55, 1973.

\bibitem{Bellettini}
C.~Bellettini.
\newblock Semi-calibrated 2-currents are pseudoholomorphic, with applications.
\newblock {\em Bull. Lond. Math. Soc.}, 46(4):881--888, 2014.

\bibitem{Chang}
S.~X. Chang.
\newblock Two-dimensional area minimizing integral currents are classical
  minimal surfaces.
\newblock {\em J. Amer. Math. Soc.}, 1(4):699--778, 1988.

\bibitem{Courant40}
R.~Courant.
\newblock The existence of minimal surfaces of given topological structure
  under prescribed boundary conditions.
\newblock {\em Acta Math.}, 72:51--98, 1940.

\bibitem{De-Giorgi55}
E.~{De Giorgi}.
\newblock {Nuovi teoremi relativi alle misure $(r-1)$-dimensionali in uno
  spazio ad $r$-dimensioni}.
\newblock {\em Ricerche Mat.}, 4:95--113, 1955.

\bibitem{DG}
E.~De~Giorgi.
\newblock {\em Frontiere orientate di misura minima}.
\newblock Seminario di Matematica della Scuola Normale Superiore di Pisa,
  1960-61. Editrice Tecnico Scientifica, Pisa, 1961.

\bibitem{DDHM}
C.~{De Lellis}, G.~{De Philippis}, J.~{Hirsch}, and A.~{Massaccesi}.
\newblock {On the boundary behavior of mass-minimizing integral currents}.
\newblock {\em arXiv e-prints}, page arXiv:1809.09457, Sep 2018.

\bibitem{DSS2}
C.~{De Lellis}, E.~{Spadaro}, and L.~{Spolaor}.
\newblock {Regularity theory for $2$-dimensional almost minimal currents I:
  Lipschitz approximation}.
\newblock {\em ArXiv e-prints. To appear in {Trans. Amer. Math. Soc.}}, Aug.
  2015.

\bibitem{DSS4}
C.~{De Lellis}, E.~{Spadaro}, and L.~{Spolaor}.
\newblock {Regularity theory for $2$-dimensional almost minimal currents III:
  blowup}.
\newblock {\em ArXiv e-prints. To appear in {Jour. of Diff. Geom}}, Aug. 2015.

\bibitem{DSS3}
C.~De~Lellis, E.~Spadaro, and L.~Spolaor.
\newblock Regularity {T}heory for 2-{D}imensional {A}lmost {M}inimal {C}urrents
  {II}: {B}ranched {C}enter {M}anifold.
\newblock {\em Ann. PDE}, 3(2):3:18, 2017.

\bibitem{DSS1}
C.~De~Lellis, E.~Spadaro, and L.~Spolaor.
\newblock Uniqueness of tangent cones for two-dimensional almost-minimizing
  currents.
\newblock {\em Comm. Pure Appl. Math.}, 70(7):1402--1421, 2017.

\bibitem{DierkesHildebrandtTromba10}
U.~Dierkes, S.~Hildebrandt, and A.~J. Tromba.
\newblock {\em Regularity of minimal surfaces}, volume 340 of {\em Grundlehren
  der Mathematischen Wissenschaften [Fundamental Principles of Mathematical
  Sciences]}.
\newblock Springer, Heidelberg, second edition, 2010.
\newblock With assistance and contributions by A. K\"{u}ster.

\bibitem{Douglas}
J.~Douglas.
\newblock Minimal surfaces of higher topological structure.
\newblock {\em Ann. of Math. (2)}, 40(1):205--298, 1939.

\bibitem{Fed}
H.~Federer.
\newblock {\em Geometric measure theory}.
\newblock Die Grundlehren der mathematischen Wissenschaften, Band 153.
  Springer-Verlag New York Inc., New York, 1969.

\bibitem{FF}
H.~Federer and W.~H. Fleming.
\newblock Normal and integral currents.
\newblock {\em Ann. of Math. (2)}, 72:458--520, 1960.

\bibitem{Fleming56}
W.~H. Fleming.
\newblock An example in the problem of least area.
\newblock {\em Proc. Amer. Math. Soc.}, 7:1063--1074, 1956.

\bibitem{Gulliver91}
R.~Gulliver.
\newblock A minimal surface with an atypical boundary branch point.
\newblock In {\em Differential geometry}, volume~52 of {\em Pitman Monogr.
  Surveys Pure Appl. Math.}, pages 211--228. Longman Sci. Tech., Harlow, 1991.

\bibitem{GulliverLesley73}
R.~Gulliver and F.~D. Lesley.
\newblock On boundary branch points of minimizing surfaces.
\newblock {\em Arch. Rational Mech. Anal.}, 52:20--25, 1973.

\bibitem{Gulliver73}
R.~D. Gulliver, II.
\newblock Regularity of minimizing surfaces of prescribed mean curvature.
\newblock {\em Ann. of Math. (2)}, 97:275--305, 1973.

\bibitem{HS}
R.~Hardt and L.~Simon.
\newblock Boundary regularity and embedded solutions for the oriented {P}lateau
  problem.
\newblock {\em Ann. of Math. (2)}, 110(3):439--486, 1979.

\bibitem{Jost85}
J.~Jost.
\newblock Conformal mappings and the {P}lateau-{D}ouglas problem in
  {R}iemannian manifolds.
\newblock {\em J. Reine Angew. Math.}, 359:37--54, 1985.

\bibitem{Osserman70}
R.~Osserman.
\newblock A proof of the regularity everywhere of the classical solution to
  {P}lateau's problem.
\newblock {\em Ann. of Math. (2)}, 91:550--569, 1970.

\bibitem{Shiffman39}
M.~Shiffman.
\newblock The {P}lateau problem for minimal surfaces of arbitrary topological
  structure.
\newblock {\em Amer. J. Math.}, 61:853--882, 1939.

\bibitem{Spivak}
M.~Spivak.
\newblock {\em Calculus on manifolds. {A} modern approach to classical theorems
  of advanced calculus}.
\newblock W. A. Benjamin, Inc., New York-Amsterdam, 1965.

\bibitem{TomiTromba88}
F.~Tomi and A.~J. Tromba.
\newblock Existence theorems for minimal surfaces of nonzero genus spanning a
  contour.
\newblock {\em Mem. Amer. Math. Soc.}, 71(382):iv+83, 1988.

\bibitem{White97}
B.~White.
\newblock Classical area minimizing surfaces with real-analytic boundaries.
\newblock {\em Acta Math.}, 179(2):295--305, 1997.

\end{thebibliography}

\end{document}